\documentclass{amsart}

\usepackage{ifthen}
\usepackage{graphicx}
\usepackage{amssymb}
\usepackage{enumerate}
\usepackage{url}
\newtheorem{anyprop}{Anyprop}[section]

\newtheorem{theorem}[anyprop]{Theorem}
\newtheorem{lemma}[anyprop]{Lemma}
\newtheorem{proposition}[anyprop]{Proposition}

\theoremstyle{definition}

\newtheorem{definition}[anyprop]{Definition}

\newtheorem{example}[anyprop]{Example}

\theoremstyle{remark}

\numberwithin{equation}{section}

\usepackage{amscd}
\usepackage{amssymb}
\input xy
\xyoption{all}

\begin{document}
\title[ON THE TOP. STRUC. OF THE (NON-)F.G. LOCUS OF SOME FROBENIUS ALG.]
{ON THE TOPOLOGICAL STRUCTURE OF THE (NON-)FINITELY-GENERATED LOCUS OF FROBENIUS ALGEBRAS EMERGING FROM STANLEY-REISNER RINGS}

\author[E. Gallego]{Edisson Gallego}
\author[J. D. Velez]{Juan D. V\'elez}
\author[S. D. Molina]{Sergio D. Molina-Aristizabal}
\author[J. P. Hernandez]{Juan P. Hernandez-Rodas}
\author[D. A. J. G\'omez]{Danny A. J. G\'omez-Ram\'irez}

\address{University of Antioquia, Medell\'in, Colombia}
\address{Universidad Nacional de Colombia, Escuela de Matem\'aticas, Medell\'in, Colombia.}
\address{University of Cincinnati, Ohio, USA.}
\address{Universidad Nacional de Colombia, Manizales, Colombia.}
\address{Parque Tech at the Instituci\'on Universitaria Pascual Bravo, Medell\'in, Colombia.}

\email{egalleg@gmail.com}
\email{jdvelez@unal.edu.co}
\email{sdaladierm@hotmail.com}
\email{jphernand@unal.edu.co}
\email{daj.gomezramirez@gmail.com}



\begin{abstract}
{
In this paper we study initial topological properties of the  (non-)finitely-generated locus
of Frobenius Algebra coming from Stanley-Reisner rings defined through face ideals. More specifically, we will give a partial answer to a conjecture made by M. Katzman about the openness of the finitely generated locus of such Frobenius algebras. This conjecture can be formulated in a precise manner:
Let us define 

\begin{equation*}
U=\{P\in \text{Spec}(R)~:~\mathcal{F}(E_{R_{p}})~\text{is a finitely
generated}~R_{p}\text{-algebra}\}
\end{equation*}%
where $\mathcal{F}$ denotes the Frobenius functor and $E_{R_{p}}$ the
injective hull of the residual field of the local ring $(R_{p},pR_{p})$.
Is the locus $U$ an open set in the Zariski Topology? In the case where $R$ is a ring of the form $R=K[[x_{1},\dots
,x_{n}]]/I,$~ where~ $I\subset R$ is a face ideal, i.e., \emph{square-free
monomial ideal}, we show that the corresponding $U$ has non-empty interior. Even more, we prove that in general $U$ contains two different types of opens sets and that in specific situation its complement contains intersections of opens and closes sets in the Zariski topology. Furthermore, we explicitly verify in some non-trivial examples that $U$ is an non-trivial open set. 
}
\end{abstract}

\maketitle


\smallskip


\section*{Introduction}

The notion of Frobenius algebra occupies a quite outstanding role in module theory and in some parts of representation theory. More specifically a Frobenius algebra can be obtained as a finite-dimensional associative algebra $A$ over a field $k$, together with a nondegenerate bilinear form $\alpha:A\times A\rightarrow K$ that commutes (regarding the order of the operations) in a compatible way with the product of the algebra $``\cdot"$, i.e. for all $x,y,z\in A$, $(x\cdot y)\alpha z=x\alpha (y\cdot z)$ (to obtain a deeper intuition about this property, the reader can consider as an enlightening example the equivalent compatibility property that the vector product ($\alpha$) and the scalar product ($\cdot$) fulfills in the $3-$dimensional real space, i.e., both expressions computes the volume of the corresponding parallelepipede).

It is an elementary verification to show that one can generate natural Frobenius algebras starting with an associative algebra $B$, equipped with a $k-$linear compatible product $\mu:B\times B\rightarrow B$ and a $k-$linear unit map $\eta:B\rightarrow k$ (satisfying all the natural properties), considering the corresponing dual Frobenius coalgebra (induced by dualizing over $k$ all the operations and doing the corresponding identifications, e.g. $ B^*\cong B$), and defining the Frobenius form $\sigma$ as the composition of the former identification with the dual of the unit map and with algebra product $\mu$. 

Finally, the nondegeneration of $\sigma$ is the only property that one needs to check by hand. One of simplest examples are the complex numbers as real vector space with the natural multiplication map and the inclusion as unit map. In this case the Frobeniues form corresponds to the composition of the  multiplication operation between complex numbers and the real part function \cite{skowronski}.

In this article, we will focus on a very special kind of Frobenius algebras, i.e., certain kinds of graduated homomorphisms' algebras over (special sorts of) quotients of polynomial rings and, additionaly, we study their finitely-generated locus. 

More specifically, in \cite{lyubeznik} G. Lyubeznik and K. E. Smith stated a conjecture about the finite generation of some Frobenius algebras, which was negatively solved by M. Katzman in \cite{Katzmanparameter}. Afterwards, in \cite{montaner}, the authors, based on \cite{Katzmanparameter}, were able to show that the Frobenius algebra of the inyective hull of a complete Stanley-Reisner ring is either principaly or infitelly generated. 

In this paper, we use some techniques developed in \cite{montaner} to prove that the topological interior of the finitely-generated locus of the Frobenius algebra of the inyective hull of the residual field of a quotient of the polynomial ring in finitelly-many variables by a square-free ideal is nonempty. 
This work is partially based on the results shown firstly in \cite[Ch.2]{gallegothesis}.

In the next section we will introduce the non-specialized reader to some of the most technical notion needed along the article.

\section{Some Important Preliminaries}

Let $R$ be a ring of characteristic $p,$ and $e\geq 0$ be any integer. For
any $R$-module $M$ we define~ $\mathcal{F}^{e}(M)=\text{Hom}^{e}(M,M),$~ the
set of additive maps from~ $M$~ to~ $M$~ which are~\emph{\ }$r^{p^{e}}$\emph{-linear}%
; that is,~ $\varphi $ is in $\mathcal{F}^{e}(M)$~ if it satisfies~ $\varphi
(rx)=r^{p^{e}}\varphi (x),$~ for~ all $r\in R$,~ $x\in M$. Also note that $\varphi \in \mathcal{F}^{e}(M),$~ and~ $\psi \in
\mathcal{F}^{e^{\prime }}(M)$,~ the map~ $\varphi \cdot \psi =\varphi \circ
\psi $ is an element of $\mathcal{F}^{e+e^{\prime }}(M)$. \newline
Consider the $e$-th Frobenius homomorphism~ $f^{e}:R\rightarrow R,$~ given
by~ $f^{e}(r)=r^{p^{e}},$~ and let~ $F_{\ast }^{e}R$~ denote the ring~ $R$~
with the product given by the Frobenius morphism. That is,~ $x\cdot
r=r^{p^{e}}x$, for~ $r\in R$,~ $x\in F_{\ast }^{e}R$.\newline
When $e=1$, the \emph{Frobenius skew polynomial ring} over $R$, $R[x,f]$, is
the left $R$-module freely generated by $(x^{i})_{i\in \mathbb{N}}$. That
is, it consists of all polynomial $\sum r_{i}x^{i},$ but with multiplication
given by the rule: $xr=f(r)x=r^{p}x$ for all $r\in R$, see \cite{sharp}. For
any ~$R$-module~ $M$,~ we define~ $F^{e}(M)=F_{\ast }^{e}R\otimes _{R}M,$~
where~ $F_{\ast }^{e}R$~ is regarded as a right~ $R$-module. For example,~
\begin{equation*}
r^{p^{e}}x\otimes m=x\cdot r\otimes m=x\otimes rm.
\end{equation*}
We think of~ $F^{e}(M)$~ as a left~ $R$-module with the product~ $r\cdot
(s\otimes m)=rs\otimes m.$ The functor $F^{e}(-)$~ is called \emph{the }$e$%
\emph{-th Frobenius functor.}\newline
In what follows, we will use the following natural identification:
\begin{equation}
\mathcal{F}^{e}(M) \cong \textrm{Hom}_{R}(F^{e}(M),M)  \label{eq1}
\end{equation}%
Given~ $\varphi \in \mathcal{F}^{e}(M),$~ we may consider~ $\psi \in \text{%
Hom}_{R}(F^{e}(M),M),$~ defined as~ $\psi (r\otimes m)=r\varphi (m)$.~ And,
reciprocally, if we have~ $\psi \in \text{Hom}_{R}(F^{e}(M),M),$ we may~
define~ $\varphi \in \mathcal{F}^{e}(M)$~ as~ $\varphi (m)=\psi (1\otimes m)$. \\*
Now recall, if $M\subset E$ are $R$ modules, then we say that $E$ is an essential extension of $M$ if every nonzero submodule of $E$ intersects $M$ nontrivially.

As it is showed in  \cite[Proposition A3.10]{eisenbud}, if $M\subset F$ is an arbitrary extension of $M$, then there is a maximal essential extension $M\subset E\subset F$. Moreover, if $F$ is an injective $R$ module then, this essencial extension $E$ is unique up to isomorphism of $R$ modules, it is denoted as $E=E(M)$ and it is called \textbf{the injective hull of $M$}.
If $R$ denotes a commutative ring with unity and characteristic $p$, $I$ is an ideal of $R$, and $n$ is a natural number, then \emph{the} $n-$\emph{th Frobenius power} of $I$ is defined as $R-$ideal

 \[I^{[p^{n}]}=\{a^{p^{n}}|~a\in I \}R.\]

\begin{definition}
Let~ $(R,m)$~ be a local ring of characteristic~ $p>0,$~ and let us denote~
by $E=E_{R}(R/m)$~ the injective hull of the residue field \cite{eisenbud}. We
define the Frobenius algebra of~ $E$~ as
\begin{equation*}
\mathcal{F}(E)=\bigoplus_{e \geq 0} \mathcal{F}^{e}(E).
\end{equation*}
\end{definition}

We note that~ $\mathcal{F}(E)$~ is a~ $\mathbb{N}$ - graded algebra over~ $%
\mathcal{F}^{0}(E)$,~ where given~ $\varphi \in \mathcal{F}^{e}(E),$~ and~ $\psi \in
\mathcal{F}^{e^{\prime }}(E)$,~ the map~ $\varphi \cdot \psi =\varphi \circ
\psi $ is an element of $\mathcal{F}^{e+e^{\prime }}(E),$~ since~ $\varphi
\circ \psi (rx)=\varphi (r^{p^{e^{\prime }}}\psi (x))=r^{p^{e+e^{\prime
}}}\varphi \circ \psi (x)$.\newline
Now,~ $\mathcal{F}^{0}(E)=\text{Hom}^{0}(E,E)=\text{Hom}_{R}(E,E)=E^{\vee }.$%
~ It is well known that when~ $(R,m)$~ is a complete local ring, the Matlis
dual of~ $E$~ is isomorphic to~ $R$, see \cite{bruns}.
Therefore, when~ $(R,m)$~ is a complete local ring we have that~ $%
\mathcal{F}^{0}(E)=R, $~ and, hence,~ $\mathcal{F}(E)$~ is an~ $R$-algebra.\newline
~\newline

\section{Openness of the finitely generated locus of the Frobenius Algebra}

The following natural question arises: For a complete local ring~ $(R,m)$,~
is the Frobenius algebra~ $\mathcal{F}(E)$~ a finitely generated~ $R$%
-algebra? The answer in general is \textit{no}, as shown in \cite{Katzman}.
There, it is shown that for the complete local ring~ $R=K[[x,y,z]]/(xy,xz),$%
~ where~ $K$~ is a field of prime characteristic, the Frobenius algebra~ $%
\mathcal{F}(E)$~ is not finitely generated as an $R$-algebra.\newline
~In spite of that negative result, another interesting question may be
posed: Is the locus~$\ $%
\begin{equation*}
U=\{P\in \text{Spec}(R)~:~\mathcal{F}(E_{R_{p}})~\text{is a finitely
generated}~R_{p}\text{-algebra}\}
\end{equation*}
open in the Zariski Topology? This seems to be a very difficult question.
\textit{We will prove the above conjecture in the case of a ring of the
form~ }$R=K[x_{1},\dots ,x_{n}]/I,$\textit{~ where~ }$I\subset K[x_{1},\dots
,x_{n}]$\textit{~ is a square-free monomial ideal.}

Before we tackle this problem, we recall the following standard facts and definitions:

First, let $A$ be a commutative ring and $I,J\subset A$ two ideals, then the \textbf{colon ideal} $(J:_{A}I)$ (or simply $(J:I)$) represents the set of all elements $a\in A$ such that $aI\subset J$. If $A$ has characteristic $p$, let us
define $F_e=(I^{[p^{e}]}:I)$ and
\begin{equation*}
F=\underset{e\geq 0}{%
\bigoplus }F_{e}f^{e},
\end{equation*}%
the $\mathbb{N}$-graded $A$-algebra with the following product: for $%
uf^{e}\in F_{e}f^{e},$ and $u^{\prime }f^{e^{\prime }}\in F_{e^{\prime
}}f^{e^{\prime }}$, we define $uf^{e}u^{\prime }f^{e^{\prime }}=u(u^{\prime
p^{e}}f^{e+e^{\prime }})$. Let 

\[L_{e}=\underset{~}{\sum }%
F_{e_{1}}F_{e_{2}}^{[p^{e_{1}}]}\cdots F_{e_{s}}^{[p^{e_{1}+\cdots
+e_{s-1}}]},\]

with $1\leq e_{1},\ldots ,e_{s}<e$ and $e_{1}+\cdots +e_{s}=e$
, and let us define $F_{<e}$ as the subalgebra of $F$ generated by $%
F_{0}f^{0},\ldots ,F_{e-1}f^{e-1}$ \cite{Katzman}.

Second, let~ $(R,m)$~be a complete local ring of prime characteristic $p,$ and let $%
S=R/I$ be a quotient by some ideal $I\subset R$. We denote~ $%
E_{R}=E_{R}(R/m),$~ and~ $E_{S}=E_{S}(S/mS)$.~ Then, the injective hull of~ $%
S/mS$~ can be obtained as~ $\text{Hom}_{R}(S,E_{R})\cong \text{Ann}_{E_{R}}I$
. Therefore,~ $E_{S}=\text{Ann}_{E_{R}}I\subset E_{R}$, see \cite[Lemma 3.2]{Notashochster}. \newline
For a better understanding of the reading we include the proofs of the following two results:
\begin{proposition}
(\cite{Katzmanparameter}, Proposition 4.1) With the above
notation.
\begin{equation*}
\mathcal{F}^{e}(E_{S})\cong \frac{(I^{[p^{e}]}~:~I)}{I^{[p^{e}]}},
\end{equation*}%
and therefore
\begin{equation*}
\mathcal{F}(E_{S})= \bigoplus_{e \geq 0}
\frac{(I^{[p^{e}]}~:~I)}{I^{[p^{e}]}}f^{e},
\end{equation*}%
where the multiplication on the right hand side is given by $xf^{e}\cdot
yf^{e^{\prime }}=xy^{p^{e}}f^{e+e^{\prime }}.$
\end{proposition}

\begin{proof}
Take~ $\varphi \in \mathcal{F}^{e}(E_{S})$.~ By~ $(\ref{eq1})$~, we
may think of~ $\varphi$~ as an element of~
$\textrm{Hom}_{R}(F^{e}(E_{S}),E_{S})$.\\
Applying the duality functor ~ $\vee =
\textrm{Hom}_{R}(\underline{~~~} , E_{R})$ to~ $\varphi$~, we get~
$\varphi^{\vee}: E_{S}^{\vee} \rightarrow F^{e}(E_{S})^{\vee}
\cong F^{e}(E_{S}^{\vee})$, (see \cite{lyubeznikfmodules}, Lemma 4.1, for
the last isomorphism).~ By Matlis duality we have~
$E_{S}^{\vee}=(S^{\vee})^{\vee} \cong S$~. Moreover, it is not difficult
to see that~ $F^{e}(S)=F^{e}(R/I) \cong
R/I^{[p^{e}]}$.\\
Therefore, we may identify~ $\varphi^{\vee}$~ with a map~
$\varphi^{\vee}: R/I \rightarrow R/I^{[p^{e}]}$~. Now, if~
$\overline{u}=\varphi^{\vee}(\overline{1})$~ then the homomorphism~
$\varphi^{\vee}$~ is just multiplication by~ $u$.~ We note that since
 ~ $\varphi^{\vee}$~ is well defined, this implies that~ $u \in
(I^{[p^{e}]}~:~I)$.\\
Define the~ $R$-homomorphism~ $\lambda: (I^{[p^{e}]}~:~I)
\rightarrow \mathcal{F}^{e}(E_{S})$~ as~ $\lambda(u)=\psi^{\vee}$~,
where~ $\psi: R/I \rightarrow R/I^{[p^{e}]}$~ is the homomorphism
given by multiplication by~ $u$.\\
The~ $R$-homomorphism~ $\lambda$~ is surjective as we saw above;
and clearly, ~ $\textrm{Ker}(\lambda)=I^{[p^{e}]}$.\\
Therefore,~ $\mathcal{F}^{e}(E_{S}) \cong
(I^{[p^{e}]}~:~I)/I^{[p^{e}]}$.
\end{proof}

\begin{lemma}
\label{generador principal}(\cite{montaner}, Lemma 2.2) With the same
notation as above. Suppose there is an element $u\in R$ such that for all $%
e\geq 0$
\begin{equation*}
(I^{[p^{e}]}:_{R}I)=I^{[p^{e}]}+(u^{p^{e}-1}).
\end{equation*}%
Then, there is an isomorphism of $S$ algebras $\mathcal{F}(E_{S})\cong
S[u^{p-1}\theta ,f]$. Here, $S[u^{p-1}\theta ,f]$ denotes the $1$-th skew
polynomial ring in the variable $u^{p-1}\theta $, (see \cite{sharp}, page
285).
\end{lemma}

\begin{proof}
We have: \[\mathcal{F}(E_{S})= \bigoplus_{e \geq 0}
\frac{(I^{[p^{e}]}~:~I)}{I^{[p^{e}]}}f^{e}=\bigoplus_{e \geq 0}
\frac{I^{[p^{e}]}+ (u^{p^{e}-1})}{I^{[p^{e}]}}f^{e}=\bigoplus_{e
\geq 0}(u^{p^{e}-1})f^{e}.\]

On the other hand: \[S[u^{p-1}\theta,f]:=\]

\[S \oplus S u^{p-1}\theta
\oplus S (u^{p-1}\theta)^{2} \oplus S (u^{p-1}\theta)^{3}\oplus
\cdots = S \oplus S u^{p-1}\theta \oplus S u^{p^{2}-1}\theta^{2}
\oplus S u^{p^{3}-1}\theta^{3}\oplus \cdots
\]
Define the $S$ homomorphism
\[\Psi:\underset{e\geq0}{\bigoplus}(u^{p^{e}-1})f^{e}\longrightarrow
\underset{e\geq0}{\bigoplus}S(u^{p-1}\theta)^{e}
\] by $\Psi(su^{p^{e}-1}f^{e})=s(u^{p-1}\theta)^{e}=su^{p^e-1}\theta^e$. This is clearly a well-defined isomorphism of $S$-algebras.

\end{proof}

Now, let $R=k[[x_{1},\ldots ,x_{n}]]$ be the formal power series ring where $%
k$ denotes a field of characteristic $p>0$ and $I\subset R$ is a square-free
monomial ideal. Then, its minimal primary decomposition $I=I_{\alpha
_{1}}\cap \cdots \cap I_{\alpha _{s}}$ is given in terms of face ideals.
That is, if $\alpha =(a_{1},\ldots ,a_{n})\in \{0,1\}^{n},$ then $I_{\alpha
}=(x_{i}|a_{i}\neq 0)$. Suppose that $\underset{}{\text{ the sum of ideals }%
\sum_{1\leq i\leq s}}I_{\alpha _{i}}$ is equal to $%
(x_{1}^{b_{1}},x_{2}^{b_{2}},\ldots ,x_{n}^{b_{n}}),$ where $\beta
=(b_{1},\ldots ,b_{n})\in \{0,1\}^{n}.$ Let us abbreviate $%
x_{1}^{b_{1}}x_{2}^{b_{2}}\cdots x_{n}^{b_{n}}$ by $x^{\beta }.$ In (\cite%
{montaner}, Proposition 3.2) the authors showed that

\begin{equation}
(I^{[p^{e}]}:_{R}I)=I^{[p^{e}]}+J_{p^{e}}+(x^{\beta })^{p^{e}-1},
\label{descripcion colon ideal en
sfree}
\end{equation}%
for any $e\geq 0$, where $J_{p^{e}}$ is either the zero ideal, or its
generators are monomials $x^{\zeta }=x_{1}^{c_{1}}\cdots x_{n}^{c_{n}}$
which satisfy $c_{i}\in \{0,p^{e}-1,p^{e}\},$ and for some $1\leq i,j,k\leq n $, we have
$c_{i}=p^{e},~c_{j}=p^{e}-1,~c_{k}=0$.

By the construction developed in \cite{montaner} one can see that by knowing $(I^{[p]}:_{R}I)$ we will readily know $%
(I^{[p^{e}]}:_{R}I),$ for any $e\geq 0$.

\begin{theorem}\label{teoremamontaner}
(\cite{montaner}, Theorem 3.5) With the previous notation and assumptions:

Let $I\subset R$ be a square-free monomial ideal, let $u=x^{\beta },$ and
let $S=R/I$. Then,

\begin{enumerate}
\item $\mathcal{F}(E_{S})\cong S[u^{p-1}\theta,f]$ is principally generated
when $J_{p}= 0$.

\item $\mathcal{F}(E_{S})$ is infinitely generated when $J_{p}\neq 0$.
\end{enumerate}
\end{theorem}

Now, we are in position to prove our new results involving topological features of the finitely generated locus of our Frobenius algebras emerging from face ideals.

\begin{theorem}\label{teorema5}
Let $R=k[[x_{1},\ldots ,x_{n}]]$ be formal power series ring in $n$ variables,
where $k$ is a field of prime characteristic $p>0$. Let us define $S=R/I,$
where $I\subset R$ is a square-free monomial ideal. Then the locus~ $%
U=\{Q\in \text{Spec}(S)~:~\mathcal{F}(E_{S_{Q}})~\text{is a
finitely generated}~S_{Q}\text{-algebra}\}$ contains the open
$U'=\text{Spec(S)}\setminus V(\text{Ann}((I^{[p]}+J_{p})/I^{[p]}))$.
\end{theorem}

\begin{proof}

Let $Q\in Spec(S)\setminus V(Ann((I^{[p]}+J_{p})/I^{[p]}))$ be a prime
ideal. Let $E_{\widehat{S_{Q}}}$ denote the injective hull of the
residue field of $\widehat{S_{Q}}$, where $\widehat{S_{Q}}$ is the
completion of  $S_{Q}$ with respect to its maximal ideal $QS_{Q}$.
A basic theoretical result, see \cite{bruns}, says
that $E_{\widehat{S_{Q}}}\cong E_{S_{Q}}$. Let us define
$M:=(I^{[p]}+J_{p})/I^{[p]}$. Note that $V(Ann(M))=supp(M)$ (see for example \cite[lemma 00L2]{stacksproject}), then $Q\notin supp(M)$, then
$0=M_{Q}=((I^{[p]}+J_{p})/I^{[p]})\otimes_{S} S_{Q}$. 

Therefore by (\ref{descripcion colon ideal en sfree}),
$((I^{[p]}:I)/I^{[p]})\otimes_{S} S_{Q}\cong
(x^{\beta})^{p-1}_{Q}$. So, we can write

\begin{equation*}
\begin{aligned}
\mathcal{F}(E_{S_{Q}})  \cong ~~~~
\mathcal{F}(E_{\widehat{S_{Q}}}) \cong \underset{e\geq 0}{
\bigoplus}
\frac{(I^{[p^{e}]}\widehat{S_{Q}}:I\widehat{S_{Q}})}{I^{[p^{e}]}\widehat{S_{Q}}}f^{e}
&\cong & \underset{e\geq 0}{\bigoplus}
\frac{(I^{[p^{e}]}:I)}{I^{[p^{e}]}}f^{e} \otimes_{S}
\widehat{S_{Q}}\\
 \cong \underset{e\geq 0}{\bigoplus}
(x^{\beta})_Q^{p^{e}-1} f^{e}\cong S_Q[x^{\beta(p-1)}\theta,f],
\end{aligned}
\end{equation*}
which is finitely generated, by Lemma \ref{generador principal}.

\end{proof}

The contention of $U'$ in $U$ in the former theorem can be strict as we will see with the following examples. Before proving this, let us state an elementary lemma of central importance in our discussion.

\begin{proposition}\label{j}

Let $k$ be a field of characteristic $p>0$, $R=k[x_1,\cdots,x_n]$, and let $I$ be a face ideal in $R$. Then 

\[(I^{[p]}:I)=I^{[p]}+J_{p}+((x^{\beta})^{p-1}),\]
where $x^{\beta}=\prod_{i=1}^nx_i^{\beta_i}$, $\beta=(\beta_1,\cdots,\beta_n)\in \{0,1\}^n$, and either $J_p=(0)$ or $J_p=(m_1,\cdots,m_r)$ and for any $\mu\in \{1,\cdots,r\},$ $m_{\mu}=\prod_{j=1}^nx_j^{d_{(\mu),j}}$ and $d_{(\mu),j}$ is equal to either $p$, $p-1$ or $0$. Furthermore, in the second case for each $m_{\mu}$ there exists $j_{\mu}$ and $i_{\mu}$ such that $d_{(\mu),i_{\mu}}=p$ and $d_{(\mu),j_{\mu}}=p-1$.

Moreover, for any $e>0$

\[(I^{[p^e]}:I)=I^{[p^e]}+J_p^{(e)}+((x^{\beta})^{p^e-1});\]
where $J_p^{(e)}=(m_1^{(e)},\cdots,m_r^{(e)})$, and for any $\mu\in\{1,\cdots,r\},$ $m_{\mu}^{(e)}=\prod_{j=1}^nx_j^{{d^{(e)}_{(\mu),j}}}$ and $d_{(\mu),j}^{(e)}$ is equal to either $p^e$ (when $d_{(\mu),j}=p$), $p^e-1$ (when $d_{(\mu),j}=p-1$) or $0$ (when $d_{(\mu),j}=0$).

\end{proposition}
\begin{proof}
Let us write $q=p^e$ for any $e>0$. First note that since $char(k)=p>0$ then if $I=(x^{\gamma_1},\cdots,x^{\gamma_z})$, for some squarefree monomials $x^{\gamma_l}$, then  $I^{[q]}=((x^{\gamma_1})^{q},\cdots,(x^{\gamma_z})^{q})$.
So, for computing explicilty all the generators of $(I^{[q]}:I)$ we set a generic monomial $x^c=\prod_{\delta=1}^nx_\delta^{c_\delta}$, where conditions for the natural numbers $c_i$ need to be defined. Explicitly, we set the conditions 
\begin{equation}
  x^cx^{\gamma_i}\in((x^{\gamma_j})^q).
  \label{monomial-conditions}
\end{equation}
for $i,j\in \{1,\cdots,z\}$. Note that due to the fact that both ideals $I$ and $I^{[q]}$ are generated by monomials, then one can check that $(I^{[q]}:I)$ is also generated by monomials . Therefore, it is enough to check the conditions given in (\ref{monomial-conditions}) for identifying such generating polynomials. 
Now, for each fixed $i$ and $j$, we obtain a specific condition on the monomial $x^cx^{\gamma_i}$ to belong to $I^{[q]}$. More specifically, this condition consists on the conjunction of $z$ requirements of the form $c_\delta\geq\kappa_{(i,j)}(\delta)$, where $\kappa_{(i,j)}(\delta)$ is equal to either $q$, $q-1$ or zero. Moreover, a monomial $x^c$ belongs to $(I^{[q]}:I)$ if and only if for each $i\in\{1,\cdots,z\}$, for at least one $j\in\{1,\cdots,z\}$, the condition in (\ref{monomial-conditions}) holds, i.e., the corresponding $z$ inequalities $c_{\delta}\geq\kappa_{(i.j)}(\delta)$ hold. So, for each fixed $i$ we can choose among $z$ possibilities to place the monomial $x^cx^{\gamma_i}$ into $I^{[q]}$, namely, with each generator of $I^{[q]}$. Thus, in general a monomial $x^c$ belongs to $(I^{[q]}:I)$ if and only if we choose for each $i$ a specific condition as in (\ref{monomial-conditions}) and we compute the whole conjunction ($\wedge$) of all the resulting inequalities. 

Note that when computing and simplifying these conjunctions, one find either identical conditions for a specific variable, or different conditions, like, for example, $c_1\geq 0 \wedge c_1\geq q\wedge c_1 \geq q-1$, which is  equivalent to $c_1\geq max(\{0,q-1,q\})= q$. So, one obtains the optimal monomial by taking the whole collection of equalities in the $c_{\delta}$. Now, with this particular $c$ one construct a generator of $(I^{[q]}:I)$. In fact, all the (monomial) generators of $(I^{[q]}:I)$ are generated in this way. Note, that there are $z^2$ possibilities to check (one by each generator of $I$ and each generator of $I^{[q]}$). When one compute a particular case for an explicit $I$, one can obtain the same generator repeated several times. 

The following step after obtaining the $z^2$ generator of $(I^{[q]}:I)$ is to eliminate the redundant ones. Next, we group these monomials in three groups. The first one is the sub-collection of monomials generating $I^{[q]}$. The second one is the monomial generated by $(x^{\beta})^{q-1}$, where $\beta_w=1$ if and only if the variable $x_w$ appears in some generator of $I$. And, the third ideal generated by the remaining monomials will be denoted by $J_q^{e}$. 

Finally, due to the fact that the collection of conjunctions giving rise to the monomials of $(I^{[q]}:I)$ are structurally independent of the value of $q$ and due to the form of these conditions we deduce both statements of our proposition.
\end{proof}
The former proposition is a more explicit variation of one of the results obtained in \cite[\S 3.1]{montaner}.
\begin{proposition}

Let $k$ be a field of characteristic $p>0$, $R=k[[x_1,x_2,x_3]]$, and let $I=(x_1x_2,x_2x_3)$.
Then $U'\subsetneq U$.
\end{proposition}

\begin{proof}

Following the method in the former proposition, we can check that for each $e>0$, 
\begin{equation}
    J_p^{(e)}=(x_1^{p^e}x_2^{p^e-1},x_2^{p^e-1}x_3^{p^e}).
\end{equation}

Again, using the same strategy of the proof of the former proposition, one can compute explicitly the generators of $Ann(M)$, where $M=(I^{[p]}+J_p)/I^{[p]}$. So, one verifies that $Supp(M)=V(Ann(M))=V((x_2))\cap V(I)$.  Therefore $U'=Spec(S)\setminus (V((x_2)) \cap V(I))=D(x_2)\cap V(I)$. 
On the other hand, one can see that 

$U=(D(x_2)\cup D(x_1x_3) \cup (V((x_1,x_2))\cap D(x_3)) \cup (V((x_2,x_3))\cap D(x_1)))\cap V(I).$

One verify this by localizing at suitable primes belonging to each of the corresponding subsets of $Spec(R)$ (the same applies to the complement of this set) and by checking that each time that the reduction of $J_p^{(e)}$ is contained (or is not) in the reduction of $(I^{[p^e]}+((x^{\beta})^{p^e-1})$. So, in each case one can mimic the argument given in theorem \ref{teorema5}, to show that $\mathcal{F}(E_{S_Q})$ is the finitely generated $S_Q$-algebra $S_Q[x^{\beta(p-1)}\theta,f]$.
Thus, one can immediately verify that $U'\subsetneq U$ in $Spec(S)$, because $(x_2)\in U$ but $(x_2)\notin U'$.

Now, it is a straightforward verification to see that 

\begin{equation}
    U=(D(x_1)\cup D(x_2) \cup D(x_3))\cap V(I).
\end{equation}
In fact, $Spec(S)\setminus U=V((x_1,x_2,x_3))=\{(x_1,x_2,x_3)\}$, which is the maximal ideal of $S$. So, $S_{(x_1,x_2,x_3)}\cong S$. Thus, $\mathcal{F}(F_{S_{(x_1,x_2,x_3)}})$ is infinitely generated by Theorem (\ref{teoremamontaner}).
So, $U$ is an open subset of $Spec(S)$.
\end{proof}

Let us denote $R'=k[[x_1,\cdots,x_n]]$
For any $d=(d_1,\cdots,d_n)\in \{0,1\}^n$, let us define its complement as $d^+=(d_1^+,\cdots,d_n^+)$
where $d_i^+=1$ if and only if $d_i=0$. Moreover, define the support of a $d\in \{0,1\}^n$ as $suppo(d)=\{x_i:d_i=1\}$, and for $d\neq (0)^n$, $x^d=\prod_{x_i\in suppo(d)}x_i$; and for $d=(0)^n$, $x^d=1$.
Also, let us define

\[G_d=D(x^d)\cap V((x_j:x_j\in suppo(d^+)))\subseteq Spec(R').
\]

Note that the collection $\{G_d\}_{d\in \{0,1\}^n}$ forms a partition of $Spec(R')$ into subsets which are intersections of open and closed sets. 

Let $\phi_d:R=k[x_1,\cdots,x_n] \rightarrow R_d=k[x_j:x_j\in suppo(d)]$ be the $k-$linear homomorphism sending $x_j$ to either $x_j$ if $x_j\in suppo(d)$, or to $1$ if $x_j\in suppo(d^+)$.
If $H$ is an ideal of $R$, then the image of $H$ under $\phi_d$ is denoted by $H_d$ and the image of a polynomial $g\in R$ is denoted similarly as $g_d$.

\begin{theorem}\label{direct}

Let $k$ be a field of characteristic $p>0$ and $I\subseteq R=k[x_1,\cdots,x_n]$ be a face ideal, and $S=R/I$. Let

\[(I^{[p]}:I)=I^{[p]}+J_p+((x^{\beta})^{p-1});\]

where $J_p=(m_1,\cdots,m_r)$, as in Proposition (\ref{j}).
Suppose that there exists a $d\in \{0,1\}^n$ such that

\begin{equation}
(I^{[p]}:_RI)_d=(I_d^{[p]}:_{R_d}I_d)=I_d^{[p]}+(J_p)_d+((x^{\beta}_d)^{p-1})=I_d^{[p]}+((x^{\beta'})_d^{p-1}),  
\end{equation}
where $\beta,\beta'\in \{0,1\}^n.$
Then the open set $D(x^{d^+})\subseteq Spec(\widehat{S})$ is contained in 

\[U=\{Q\in \text{Spec}(\widehat{S})~:~\mathcal{F}(E_{\widehat{S_{Q}}})~\text{is a
finitely generated}~\widehat{S_{Q}}\text{-algebra}\}.\]
\end{theorem}

\begin{proof}
Let $Q\in D(x^{d^+})$. Due to the fact that all the variables $x_j$ in $x^{d^+}$ are unities in $S_Q$, from the hypothesis we see that

\[(I^{[p]}:I)\otimes \widehat{S_Q}=(I^{[p]}+((x^{\beta'})^{p-1}))\otimes \widehat{S_Q}.\]
Now, by proposition (\ref{j}), we can extend the former fact to any $e>0$ as follows

\[(I^{[p^e]}:I)\otimes \widehat{S_Q}=(I^{[p^e]}+((x^{\beta'})^{p^e-1}))\otimes \widehat{S_Q}.\]

So, doing a reasoning essentially identical as the proof of theorem (\ref{teorema5}) we verify that 

\[\mathcal{F}(E_{\widehat{S_Q}})\cong \widehat{S_Q}[x_d^{\beta'(p-1)}\theta,f].
\]

In conclusion, $Q\in U$.

\end{proof}

Now, we prove a kind of dual version of the former theorem giving a sufficient condition for a set of the form $G_d$ not to belong to $U$.

\begin{theorem}\label{complement}

Let $k$ be a field of characteristic $p>0$ and $I\subseteq R=k[x_1,\cdots,x_n]$ be a face ideal, and $S=R/I$. Let

\[(I^{[p]}:I)=I^{[p]}+J_p+((x^{\beta})^{p-1});\]

where $J_p=(m_1,\cdots,m_r)$, as in prop. (\ref{j}).

Suppose that there exists a $d\in \{0,1\}^n$ such that for the prime ideal $Q_{d^+}=(x_j:x_j\in suppo(d^+))\in G_d$ it holds that there exists a monomial generator of $J_p \otimes \widehat{S_{Q_{d^+}}}$ of the form $x^{\sigma}$, where there exists $\sigma_a=0,\sigma_b=p-1$ and $\sigma_c=p$, for some indexes $a,b,c\in\{1,\cdots,n\}$. Assume that $x^{\sigma}\notin (I^{[p]}+(x^{\beta(p-1)}))\otimes \widehat{S_{Q_{d^+}}}$, then 

\[G_d\cap V(I\widehat{S})\subseteq U^c.
\]
\end{theorem}

\begin{proof}
 Let $Q$ be a prime ideal in $G_d\cap V(I\widehat{S})$. Note that by definition $Q_{d^+}$ is the minimal prime of $G_d\cap V(I\widehat{S})$, so $Q_{d^+}\subseteq Q$. Moreover when $J_p$ is not containing in $I^{[p]}+(x^{\beta(p-1)})$, we can assume without loss of generality that there exists a monomial generator $x^{\gamma}$ in $J_p$ with the configuration of exponents given in the initial condition of our theorem (see for example \cite[\S 3.1]{montaner}). So, the hypothesis is a natural fact that can be checked (resp. refused) simply by verifying that the classes of the original generators of $J_p$ in the localization belong (or not) to $(I^{[p]}+(x^{\beta(p-1)}))\otimes S_{Q_{d^+}}$.

Now, due to the fact that by localizing at $Q_{d^+}$, $x^{\sigma}\notin (I^{[p]}+(x^{\beta(p-1)}))\otimes \widehat{S_{Q_{d^+}}}$, then by localizing at $Q$, 
$x^{\sigma}\notin (I^{[p]}+(x^{\beta(p-1)}))\otimes \widehat{S_{Q}}$. This holds because in the last localization we are inverting less elements that in the first one. So, if $x^{\sigma}\in (I^{[p]}+(x^{\beta(p-1)}))\otimes \widehat{S_Q}$, then localizing again (at $Q_{d^+}$), we would obtain $x^{\sigma}\in (I^{[p]}+(x^{\beta(p-1)}))\otimes \widehat{S_{Q_{d^+}}}$.
Furthermore, by a similar proof like the one given in proposition (\ref{j}) and following the notation there, we can extend the former fact to any $e>0$ as $(x^{\sigma})^{(e)}\notin (I^{[p^e]}+(x^{\beta(p^e-1)}))\otimes \widehat{S_{Q}}$.

Due to the fact that $R_Q$ is a complete local ring containing a field of characteristic $p>0$  and that $S_Q\cong R_Q/(IS_Q)$, we can apply Katzman's criterion \cite{Katzman} exactly in the same way of \cite[Prop.3.5]{montaner} to obtain the following result: 

For $e>0$, let $F_e=(I\widehat{S_Q}^{[p^e]}:_{\widehat{S_Q}}I\widehat{S_Q})$ and define \[L_{e}=\underset{~}{\sum }%
F_{e_{1}}F_{e_{2}}^{[p^{e_{1}}]}\cdots F_{e_{s}}^{[p^{e_{1}+\cdots
+e_{s-1}}]},\] with $1\leq e_{1},\ldots ,e_{s}<e$ and $e_{1}+\cdots +e_{s}=e$
, and let us define $\mathcal{F}_{<e}$ as the subalgebra of $\mathcal{F}(E_{\widehat{S_Q}})$ generated by $%
\mathcal{F}^{0}(E_{\widehat{S_Q}}),\ldots ,\mathcal{F}^{e-1}(E_{\widehat{S_Q}})$. Then $\mathcal{F}_{<e}\cap \mathcal{F}^e(E_{\widehat{S_Q}})=L_e$.
Moreover, we verify exactly like in \cite{Katzman} that for any $e>0$, $(x^{\sigma})^{(e)}\in F_e$ but $(x^{\sigma})^{(e)}\notin L_e$. So, $Q\notin U$.
In conclusion, $G_d\cap V(I\widehat{S})\subseteq U^c$.
\end{proof}

One can use the former two theorems as initial tools for computing explicitly the finitely generated locus of specific Frobenius algebras emerging from simple, but not trivial face ideals. We explore this usage in the following example

\begin{example}
Let $k$ be a field of characteristic $p>0$, $R=k[[x_1\cdots,x_4]]$, and let $I=(x_1x_2x_3,x_3x_4)$.
\end{example}

As before, we can check that for all $q=p^e$, with $e>0$.

\begin{equation}
(I^{[q]}:I)=I{[q]}+J_p^{
(e)}+((x_1x_2x_3x_4)^{q-1})
\end{equation}

where $J_p^{(e)}=(x_1^qx_2^qx_3^{q-1},x_3^{q-1}x_4^q)$.

Moreover, by dividing the topological space $Spec(R)$ into subsets of the form $D(x_{i_1}\cdots x_{i_r})\cap V(x_{i_{r+1}},\cdots,x_{r_4})$, where $\{i_1,\cdots,i_4\}=\{1,\cdots,4\}$, we can check by using both criteria described in theorems (\ref{direct}) and (\ref{complement}) that the lifting of $U$ in $Spec(R)$, say $U_R$ is the following set

\begin{equation}
U_R=D(x_3)\cup D(x_1x_2)\cup (D(x_1x_4)\cap v((x_2,x_3)))\cup(D(x2x_4)\cap V((x_1,x_3)).
\end{equation}

On the other hand, one can directly check that the complement of this set is

\[Spec(R)\setminus U_R=U_R^{c}=\]

\[V((x_3,x_1x_2,x_1x_4,x_2x_4)\cup (V((x_1,x_3)\cap D(x_2))\cup (V(x_2,x_3,x_4)\cap D(x_1))
\]

\[\cup D(x_1x_2)\cup D(x_2x_3).
\]

Now, it can be seen by elementary arguments that $U_R^c$ is not closed. So, $U_R$ is not open. However, one can compute the reduction of $U_R$ to $Spec(S)$ to obtain

\begin{equation}
U=U_R\cap V(I)=D(x_1x_3x_4),
\end{equation}

So, again, $U$ is an open set.

In order to obtain more heuristic information regarding the topological structure of the (non-)finitely generated locus of our Frobenius algebras, it is worth to do the explicit identification of $U$ (resp. $U^c$) for more complex examples. So, the following sample provides a next natural step to pursue in order to characterize the complete topological structure of $U$ (resp. $U^c$).

\begin{example}
Let $k$ be a field of characteristic $p>0$, $R=k[[x_1\cdots,x_5]]$, and let $I=(x_1x_2x_3,x_3x_4,x_4x_5)$. 
\end{example}

We can check that for all $q=p^e$, with $e>0$.

\begin{equation}
(I^{[q]}:I)=I{[q]}+J_p^{
(e)}+((x_1x_2x_3x_4x_5)^{q-1})
\end{equation}

where $J_p^{(e)}=(x_1^{q-1}x_2^{q-1}x_3^{q}x_4^{q-1},x_3^{q-1}x_4^qx_5^{q-1})$.

So, in this case we should check a larger number of $G_d$-s using Theorems (\ref{direct}) and (\ref{complement}).

Finally, the main open question to solve is to determine if $U$ is always open or not! So, one initial way to proceed is to continue the computations for several (experimental) examples (like the former ones) in order to get a deeper intuition of the (non-)validity of this query.

\section*{Acknowledgements}
The authors want to thank the Universidad Nacional of Colombia. D. A. J. G\'omez-Ram\'irez thanks the Instituci\'on Universitaria Pascual Bravo and Visi\'on Real Cognitiva S.A.S and Johan Baena for all his kindness and support. Finally, Edisson Gallego thanks sincerely to Mordechai Katzman all the inspiring discussions.

\bibliographystyle{amsplain}

\begin{thebibliography}{10}

\bibitem{bruns}
W.~Bruns and J.~Herzog, \emph{{Cohen-{M}acaulay Rings, Revised Edition}},
  Cambridge University Press, 1998.

\bibitem{stacksproject}
Aise~Johan de~Jong et~al., \emph{Stacks project. open source project}, 2010.

\bibitem{eisenbud}
D.~Eisenbud, \emph{{Commutative Algebra with a View Toward Algebraic
  Geometry}}, GTM, vol. 150, New York: Springer-Verlag, 1995.

\bibitem{gallegothesis}
E.~Gallego-Gonzalez, \emph{Conjeturas homologicas y la teor\'ia de modelos},
  Ph.D. thesis, Universidad Nacional de Colombia (Sede Medell\'in), 2012, In
  preparation.

\bibitem{Notashochster}
Melvin Hochster, \emph{Local cohomology}, Unpublished Notes.

\bibitem{Katzmanparameter}
Mordechai Katzman, \emph{Parameter-test-ideals of cohen--macaulay rings},
  Compositio Mathematica \textbf{144} (2008), no.~4, 933--948.

\bibitem{Katzman}
\bysame, \emph{A non-finitely generated algebra of frobenius maps}, Proceedings
  of the American Mathematical Society \textbf{138} (2010), no.~7, 2381--2383.

\bibitem{lyubeznikfmodules}
G.~Lyubeznik, \emph{{$F$-modules: applications to local cohomology and
  $D$-modules in characteristic $p>0$}}, J. reine angew. Math. \textbf{491}
  (1997), 65--130.

\bibitem{lyubeznik}
Gennady Lyubeznik and Karen Smith, \emph{On the commutation of the test ideal
  with localization and completion}, Transactions of the American Mathematical
  Society \textbf{353} (2001), no.~8, 3149--3180.

\bibitem{montaner}
Josep~{\`A}lvarez Montaner, Alberto~F Boix, and Santiago Zarzuela,
  \emph{{Frobenius and Cartier algebras of Stanley--Reisner rings}}, Journal of
  Algebra \textbf{358} (2012), 162--177.

\bibitem{sharp}
Rodney~Y Sharp, \emph{Big tight closure test elements for some non-reduced
  excellent rings}, Journal of Algebra \textbf{349} (2012), no.~1, 284--316.

\bibitem{skowronski}
Andrzej Skowro{\'n}ski and Kunio Yamagata, \emph{Frobenius algebras}, vol.~12,
  European Mathematical Society, 2011.

\end{thebibliography}
\providecommand{\bysame}{\leavevmode\hbox to3em{\hrulefill}\thinspace}
\providecommand{\MR}{\relax\ifhmode\unskip\space\fi MR }
\providecommand{\MRhref}[2]{%
  \href{http://www.ams.org/mathscinet-getitem?mr=#1}{#2}
}
\providecommand{\href}[2]{#2}

\end{document}